\documentclass{amsart}
\usepackage[T1]{fontenc}
\usepackage[utf8]{inputenc}

\usepackage[paperwidth=182mm,paperheight=257mm,text={130mm,205mm},centering]{geometry}
\usepackage{tikz-cd,microtype,amssymb,mathtools}

\usepackage{hyperref}

\usepackage{mathrsfs}
\usepackage{mathdots}
\usepackage{enumitem}

\numberwithin{equation}{section}
\theoremstyle{plain}
\newtheorem{theorem}{Theorem}[section]
\newtheorem{proposition}[theorem]{Proposition}
\newtheorem{lemma}[theorem]{Lemma}

\theoremstyle{definition}
\newtheorem{situation}[theorem]{Situation}

\newtheorem{remark}[theorem]{Remark}
\newtheorem{observation}[theorem]{Observation}

\renewcommand{\mathbb}{\mathbf}
\renewcommand{\setminus}{\mathbin{\rule[0.2em]{0.67em}{0.12em}}}%

\newcommand{\et}{\mathrm{\acute et}}

\DeclareMathOperator{\Aut}{\mathrm{Aut}}

\DeclareMathOperator{\Ker}{\mathrm{Ker}}
\DeclareMathOperator{\Sym}{\mathrm{Sym}}

\title[Generic triviality]{Generic triviality of automorphism groups of complete intersections}

\author{Renjie Lyu}
\address{School of Mathematical Sciences, Xiamen University, Xiamen 361005, China.}
\email{r.lyu@xmu.edu.cn}

\author{Dingxin Zhang}
\address{YMSC, Tsinghua University, Beijing 100086, China.}
\email{dingxin@tsinghua.edu.cn}


\begin{document}

\begin{abstract}
In this paper, we study the automorphism groups of polarized
algebraic varieties via cohomology and monodromy representations. 
As an application, we prove that a general smooth complete intersection 
for a number of types admits no non-trivial automorphisms. 
This result improves the work of Chen-Pan-Zhang in positive characteristic.
\end{abstract}

\maketitle

\tableofcontents

\section*{Introduction}
Let \(k\) be an algebraically closed field,
and let \(X \subset \mathbb{P}^{n}_{k}\) be a smooth closed subvariety.
A \emph{linear automorphism} is an automorphism of \(\mathbb{P}^{n}\)
that preserves \(X\). We denote by \(\mathrm{Aut}_{L}(X)\) the group of 
linear automorphisms of \(X\).

This paper aims to show that a general smooth complete intersection
in projective space admits no non-trivial linear automorphisms, with some 
exceptional cases. For hypersurfaces,
Matsumura and Monsky~\cite{Matsumura-Monsky:hypersurfaces-auto} 
proved that when \(n \geq 3\) and \(d \geq 3\), the group of linear
automorphisms for a generic degree \(d\) hypersurface in
\(\mathbb{P}^{n}\) is trivial. The cases for \(n = 2, d \geq 3\) 
is treated in~\cite[\S~10]{Katz-Sarnak:monodromy}
and~\cite{Poonen:auto-hypersurface}. 
In~\cite{Benoist:separated-moduli-ci}, Benoist showed that the 
linear automorphism group of any smooth complete intersection is finite, 
except for the case of hyperquadrics. Based on this result, 
Chen et al. proved that when \(\mathrm{char}(k)=0\), a general smooth 
complete intersection in \(\mathbb{P}^n_k\) of multidegree 
\((d_1,\ldots, d_c)\neq (2), (2,2)\) has no non-trivial linear automorphism. 
In positive characteristic \(p\), they asserted that the 
order of any non-trivial linear automorphism, if it exists, 
must be divided by \(p\), see~\cite[Thm. 1.3]{Chen-Pan-Zhang:auto-coh-ci}.
Javanpeykar and Loughran obtained the same conclusion for complex 
complete intersections of type \((2, 2, 2)\) via a different approach.


In this exposition, we introduce a general method for studying automorphisms
of polarized varieties in a. Applying this method to complete 
intersections, we obtain the following main result:

\begin{theorem}[= Theorem~\ref{theorem:gen-trivial-auto}]%
\label{theorem:main}
Let \(k\) be an algebraically closed field of characteristic
\(p \neq 2\). Suppose \((d_{1},\ldots,d_{c})\) is a sequence
of positive integers satisfying one of the following
\begin{itemize}
\item \(3 \leq d_{1} \leq d_{2} \leq \cdots \leq d_{c}\) and \(p \nmid d_{1}\);
\item \(c \geq 3\) and \(2 = d_{1} = d_{2} = d_{3} \leq d_4\leq \cdots \leq d_{c}\).
\end{itemize}
Then, for a general smooth complete intersection \(X\subset \mathbb{P}^{n}_{k}\) 
of multidegree \((d_{1},\ldots,d_{c})\) with \(c < n\), we have
\(\Aut_L(X)=\{id\}\), unless \(X\) is a cubic curve.
\end{theorem}
In positive characteristic, this result confirms the triviality of 
linear automorphisms for a general complete intersection of many types, 
which partially strengthens the result in
~\cite{Chen-Pan-Zhang:auto-coh-ci}. Other than complete 
intersections in projective space, we also study the case 
of a sufficiently ample hypersurface in a smooth projective variety.

\begin{theorem}[= Theorem~\ref{theorem:general-type-hypersurface}]%
\label{theorem:trivial}
    Let \(P\) be a smooth complex projective variety. Let \(\omega_P\)
    be the canonical sheaf on \(P\), and let \(\mathscr{L}\) be
    a very ample invertible sheaf such that \(\mathscr{L} \otimes \omega_P\)
    remains very ample. Then a general smooth section of \(\mathscr{L}\)
    has no non-trivial automorphisms.
\end{theorem}



Our proof of Theorem~\ref{theorem:main} and~\ref{theorem:trivial} are 
inspired by Katz and Sarnak~\cite[\S 10-11]{Katz-Sarnak:monodromy}, 
who showed that the monodromy group of the universal family of smooth
curves or smooth projective hypersurfaces is sufficiently large, which
leads to that the automorphism group of a general member in the family
is trivial. By considering the action of the automorphism group on cohomology, 
we generalize and refine this technique and apply it to complete intersections.

Let us consider a family of smooth projective subvarieties in the projective
space \(\mathbb{P}^{N}\) parameterized by an irreducible \(k\)-variety \(B\), 
given by the diagram:
\begin{equation*}
\begin{tikzcd}[column sep=small]
\mathcal{X} \ar[r,hook, "i"] \ar[rd,bend right=20,"\pi",swap] & \mathbb{P}^{N} \times B \ar[d] \\
{} & B.
\end{tikzcd}
\end{equation*}
Let \(\mathcal{L}:=i^*\mathcal{O}_{\mathbb{P}^N}(1)\) denote the natural
polarization. In Section~\ref{sec:criterion}, we demonstrate the 
following criterion for the triviality of the linear automorphism group
of a general fiber.

\begin{theorem}[= Theorem~\ref{theorem:mere}]
\label{theorem:intro-criterion}
Assume the following conditions for the family \(\pi:\mathcal{X}\to B\).
\begin{itemize}

\item The relative polarized automorphism group scheme 
\(\Aut_B(\mathcal{X}, \mathcal{L})\) is finite and unramifed over \(B\).

\item There exists a closed point \(t \in B\) such that 
\(\mathrm{Aut}_{L}(X_{t})\) acts faithfully on the \(\ell\)-adic 
cohomology \(\mathrm{H}^{m}_{\et}(X_{t};\mathbb{Q}_{\ell})\) 
for all \(\ell\neq \mathrm{char}(k)\), where \(m = \dim X_t\). 

\item The monodromy group associated to the local system 
\((R^{m}\pi_{\ast}\mathbb{Q}_{\ell})_\mathrm{prim}\)
is sufficiently big in the sense of Theorem~\ref{theorem:mere}\eqref{item:big-monodromy}.
\end{itemize}
Then, for a general point \(b \in B\), the linear automorphism group 
\(\mathrm{Aut}_{L}(X_{b})\) is either trivial or isomorphic to 
\(\mathbb{Z}/2\mathbb{Z}\).
\end{theorem}

In Section~\ref{sec:auto-ci}, we verify that these conditions hold
for complete intersections of types presented in Theorem~\ref{theorem:main}. 
Given a family \(\pi\colon \mathcal{X}\to B\) of smooth complete 
intersections of fixed multidegree, the group scheme 
\(\Aut_B(\mathcal{X}, \mathcal{L})\) is finite over \(B\) if the 
corresponding moduli stack is separated. For non-ruled polarized varieties, 
the separatedness follows from the classical result of
Matsusaka and Mumford's ~\cite{Matsusaka-Mumford:two-fundamental-theorems}.
Remarkably, Benoist affirmed that separatedness holds for all complete 
intersection except the hyperquadrics~\cite{Benoist:separated-moduli-ci}.

The interaction between and automorphism action on cohomology and the
monodromy group plays a central role. A ``sufficient big'' monodromy 
group constrains automorphisms. For hyperplane sections, the monodromy 
action can be characterized using vanishing cycles of Lefschetz 
pencils~\cite[\S 4-5]{Deligne:Weil-I}.

Verifying the faithfulness of the cohomological action by the 
automorphism group occupies the main part of Section~\ref{sec:auto-ci}. 
The highlight of Theorem~\ref{theorem:intro-criterion} is that 
it suffices to check faithfullness for a single fiber in the family. 
We therefore consider specific Fermat-type complete intersections~\eqref{equation:Fermat-eqs},
and show its cohomological action by the auromorphism is faithful.
Lastly, the possibility of 
\(\mathrm{Aut}(X_{b}) \simeq \mathbb{Z}/2\mathbb{Z}\) is excluded by 
explicit computation in Proposition~\ref{proposition:no-invoultions}
by assuming \(\mathrm{char}(k)\neq 2\).

Our result in Theorem~\ref{theorem:main} may admit further enhancement. 
For instance, the assumption \(p\nmid d_1\) is indispensible for the
smoothness of the Fermat-type complete intersections in positive characteristic. 
However, one may choose other examples of complete intersections to remove 
this constriant. For the types that not covered in 
Theorem~\ref{theorem:main}, see Remark~\ref{remark:ci-quadric}.

We note that~\cite{Chen-Pan-Zhang:auto-coh-ci} also established the 
faithfulness of cohomological action. While our conclusions are similar, 
the proof strategies diverge: they first prove that a general complex 
complete intersection has trivial automorphism group and deduce the 
faithfulness, whereas we prove the faithfulness directly for a specific 
class of complete intersections and then conclude the triviality 
of automorphisms for general members.

\subsection*{Acknowledgements}
We are grateful to Hsueh-Yung Lin, Yujiro Kawamata for discussions 
improving some results. 
We would like to thank Xi Chen and Wenfei Liu for comments on the manuscript.

\section{A criterion of generic triviality of automorphisms}
\label{sec:criterion}

\begin{situation}\label{situation:family-of-varieties}%
  Let \(k\) be an algebraically closed field. Consider 
  a smooth projective family of algebraic varieties
  \begin{equation*}
    \begin{tikzcd}
      \mathcal{X} \ar[rd,bend right=20, swap, "\pi"]\ar[r,hook, "i"]
      & \mathbb{P}^N\times B \ar[d] \\
      & B
    \end{tikzcd}
  \end{equation*}
  parameterized by an irreducible \(k\)-variety \(B\). Let \(\mathcal{L}\) 
  be the polarization \(i^*\mathcal{O}_{\mathbb{P}^N}(1)\). Denote by
  \(\Aut_B(\mathcal{X}, \mathcal{L})\) the group scheme that represents 
  the \(\mathcal{L}\)-polarized automorphism functor for the morphism
  \(\pi: \mathcal{X}\to B\). Let \(\mathcal{G}\) be a closed subgroup 
  scheme of \(\Aut_B(\mathcal{X}, \mathcal{L})\). Let \(X_b\) be the 
  fiber over a closed point \(b \in B\), and \(G_b\) be the fiber of 
  \(\mathcal{G}\) over \(b\). 

  Fix a prime number \(\ell\) which is invertible in \(k\). Let 
  \(\mathrm{H}_{\et}^m(X_b; \mathbb{Q}_{\ell})\) be the \(\ell\)-adic
  cohomology group. The cup-product on 
  \(\mathrm{H}_{\et}^m(X_b; \mathbb{Q}_{\ell})\) gives a non-degenerate 
  symmetric (resp. anti-symmetric) form \(\psi\) if \(m\) is even 
  (resp. odd). The primitive part 
  \(\mathrm{H}_{\et}^m(X_b; \mathbb{Q}_{\ell})_\textrm{prim}\) 
  is the orthogonal complement of the image of the restriction map 
  \(\mathrm{H}_{\et}^m(\mathbb{P}^N; \mathbb{Q}_{\ell}) \to 
  \mathrm{H}_{\et}^m(X_b; \mathbb{Q}_{\ell})\). Let \(\pi_1(B, b)\) 
  denote the \'etale fundamental group of \(B\) at the point \(b\).
  The geometric monodromy group \(G_{\mathrm{geom},\ell}\) is defined
  to be the Zariski closure of the image of the monodromy representation
  \begin{equation*}
      \rho\colon \pi_1(B,b) \to \operatorname{Aut}(\mathrm{H}_{\et}^m
      (X_b;\mathbb{Q}_{\ell})_{\mathrm{prim}}, \psi).
  \end{equation*}

  \begin{theorem}
  \label{theorem:mere}
  Use the notation in Situation~\textup{\ref{situation:family-of-varieties}}.
  Suppose that the following hypotheses hold.
  \begin{enumerate}
    \item The group scheme \(\mathcal{G} \to B\) is finite and unramifed over \(B\).
    \label{item:finiteness}

    \item For every prime \(\ell\neq \mathrm{char}(k)\), there exists 
    a closed point \(t \in B\) such that the map
    \(G_{t} \to \operatorname{Aut}(\mathrm{H}_{\et}^m(X_t;\mathbb{Q}_{\ell}))\)
    is injective.
    \label{item:generic-faithfulness}

    \item For every prime \(\ell\neq \mathrm{char}(k)\), the geometric 
    monodromy group 
    \[
    G_{\mathrm{geom},\ell}=
    \begin{cases}
    \text{the symplectic group}~ \mathrm{Sp}(\mathrm{H}_{\ell}, \psi), & \text{if}~ m ~\text{odd},\\
    \text{the orthogonal group}~ \mathrm{O}(\mathrm{H}_{\ell}, \psi), & \text{if}~ m ~\text{even},
    \end{cases} 
    \]
    where \(\mathrm{H}_{\ell}\) represents the \(\mathbb{Q}_{\ell}\)-space
    \(\mathrm{H}_{\et}^m(X_b;\mathbb{Q}_{\ell})_{\mathrm{prim}}\).
    \label{item:big-monodromy}
  \end{enumerate}
  Then \(G_b\) is either a trivial group or isomorphic to \(\mathbb{Z}/2\mathbb{Z}\)
  for a general \(b\in B\).
\end{theorem}
\end{situation}





Hypotheses~\ref{theorem:mere}(\ref{item:finiteness}--\ref{item:big-monodromy})
are not sufficient to deduce the triviality of \(G_{b}\). For instance, 
let \(B\) be the space of \(6\) distinct ordered points on \(\mathbb{P}^1\), 
and \(\mathcal{X} \to B\) is the family of hyperelliptic curves branching 
over the given points, and \(\mathcal{G} = \operatorname{Aut}_B(\mathcal{X})\).
Then every \(X_b\) has an automorphism of order \(2\), although
the finiteness of automorphism groups, the faithfulness of the action
of \(\operatorname{Aut}(X_b)\) on \(\mathrm{H}^1(X_b)\), and the
bigness of the geometric monodromy group are all satisfied.



\begin{proof}[Proof of Theorem~\ref{theorem:mere}]
  The morphism \(p\) is finite and unramified. The fiber group \(G_b\)
  over any \(b\in B\) is finite and discrete, and the order 
  \(|G_b|\) is equal to the rank of the finite \(\mathcal{O}_{B,b}\)-module
  \((p_*\mathcal{O}_{\mathcal{G}})_b\). It follows that the function 
  \(b\mapsto |G_b|\) is upper semi-continuous. Then there is a non-empty 
  open subset where the number \(|G_b|\) reaches the minimum. 
  Consider the closed subgroup scheme 
  \[
  \mathcal{Z} \coloneqq \operatorname{Ker}(\mathcal{G}\to \mathrm{GL}(R^m\pi_*\mathbb{Q}_{\ell}))
  \]
  of cohomologically trivial automorphisms. Then 
  \(\mathcal{Z} \to B\) remains a finite and unramified morphism.
  Let \(B'\subset B\) be the open subset where the order function 
  \(b \mapsto |Z_b|\) reaches the minimum. 
  By Hypothesis~\ref{theorem:mere}(\ref{item:generic-faithfulness}),
  the minimum has to be \(1\), i.e. \(\mathcal{Z}\) is the 
  identity group scheme over \(B'\). 

  By the generic flatness, there exists a non-empty open subset 
  \(U'\subset B\) such that morphism \(p\colon \mathcal{G}\to B\) is
  flat over \(U'\). Then \(p\colon \mathcal{G}|_{U'}\to U'\) is a finite 
  \'etale morphism. Since \(B\) is irreducible, \(U'\cap B'\) is nonempty. 
  Hence we have a Zarski open dense subset \(U\subset B\) 
  such that \(\mathcal{G}/U\) is a finite \'etale covering, and the 
  group \(G_b\) acts faithfully on 
  \(\mathrm{H}^m_{\et}(X_b; \mathbb{Q}_{\ell})\) for any \(b\in U\).
  Let \((V,b)\to (U,b)\) be a finite \'etale connected covering 
  such that the pullback \(\mathcal{G}|_V\) becomes a constant 
  group scheme over \(V\). Then the monodromy action of \(\pi_1(V, b)\) 
  on \(G_b\) is invariant. Viewing \(G_b\) as a subgroup of 
  \(\mathrm{GL}(\mathrm{H}_{\ell})\), we have
  \[
  \rho(\gamma)\cdot g\cdot \rho(\gamma)^{-1}=g, ~\forall g\in G_b, ~\forall \gamma\in \pi_1(V, b)
  \]
  where \(\rho\) is the monodromy action of \(\pi_1(V, b)\).

  The induced map \(\pi_1(U, b)\to \pi_1(B, b)\) is surjective. Hence 
  the corresponding geometric monodromy group \(G_{\mathrm{geom}, U}\) 
  on \(U\) remains \(\mathrm{Sp}(\mathrm{H}_{\ell}, \psi)\) or 
  \(\mathrm{O}(\mathrm{H}_{\ell}, \psi)\) from Hypothesis~\ref{theorem:mere}\eqref{item:big-monodromy}.
  The fundamental group \(\pi_1(V,b)\) is a finite index 
  subgroup of \(\pi_1(U,b)\), and the geometric monodormy group 
  \(G_{\mathrm{geom}, V}\) on \(V\) also has finite index in 
  \(G_{\mathrm{geom}, U}\). Note that the identity component of 
  \(G_{\mathrm{geom}, U}\) is connected. There exists no continuous 
  map from a connected group to a discrete group. Hence the identity 
  component \(G_{\mathrm{geom}, V}^{\circ}\) is
  \(\mathrm{Sp}(\mathrm{H}_{\ell}, \psi)\) or 
  \(\mathrm{SO}(\mathrm{H}_{\ell}, \psi)\)

  A linear automorphism on \(\mathrm{H}_{\ell}\) that commutes with all 
  symplectic or special orthogonal matrices must be a scalar action
  by some \(\lambda\in \mathbb{Q}^{*}_\ell\). If the degree of the group 
  scheme \(\mathcal{G}/U\) is \(r\), i.e., the order of \(G_b\) is \(r\). 
  Then \(\lambda\) is a \(r\)-th root of unity in \(\mathbb{Q}_{\ell}\)
  for every prime \(\ell\neq \mathrm{char}(k)\). By Hensel's lemma, 
  we know that \(\lambda\) is either an \((\ell-1)\)-th root of unity 
  for \(\ell\) odd or \(\pm 1\) for \(\ell=2\). In the case of 
  \(\mathrm{char}(k) \neq 2\), let \(\ell=2\). Then \(\lambda = \pm 1\). 
  When \(\mathrm{char}(k) = 2\) and \(r\) is odd, we can choose an 
  odd prime \(\ell\) such that \(\ell-1\) is coprime to \(r\). Then 
  \(\lambda^{\ell-1}=1\) implies \(\lambda=1\) since \((\ell-1,r)=1\). 
  When \(\mathrm{char}(k) = 2\) and \(r\) is even, choose the odd prime 
  \(\ell\) such that \(\frac{\ell-1}{2}\) is coprime to \(r\). Then 
  we have \(\lambda^{\frac{\ell-1}{2}}=\pm 1\), which implies 
  \(\lambda=\pm 1\) since \((\frac{\ell-1}{2}, r)=1\). In conclusion, 
  the order of the group \(G_b\) for \(b\in U\) is bounded by \(2\).
\end{proof}

As an application of Theorem~\ref{theorem:mere}, we show that
a sufficiently ample, sufficiently general hypersurface in any smooth
complex projective variety is free of automorphisms. To begin with, 
we need a lemma.

\begin{lemma}%
  \label{lemma:faithful-general-type}
  Let \(X\) be an m--dimensional, smooth complex projective variety
  with very ample canonical bundle. Then the natural map
  \(\operatorname{Aut}(X) \to \operatorname{Aut}(\mathrm{H}^m(X; \mathbb{Q}))\)
  is injective.
\end{lemma}

\begin{proof}
  Since the canonical bundle \(\omega_X\) is very ample, we consider
  the canonical embedding \(X \hookrightarrow \mathbb{P}(\mathrm{H}^0(X, \omega_X)^{\vee})\).
  Let \(f\) be an automorphism of \(X\) such that \(f^{\ast}\) is the
  identity.  By the Hodge decomposition \(f^{\ast}\) acts identically
  on \(\mathrm{H}^0(X,\omega_X)\), which induces the idenity on
  \(\mathbb{P}(\mathrm{H}^0(X, \omega_X)^{\vee})\).
  Via the canonical embedding, we have \(f = \operatorname{Id}_X\).
\end{proof}

\begin{theorem}%
  \label{theorem:general-type-hypersurface}
  Let \(P\) be a smooth complex projective variety. Let \(\omega_P\)
  be the canonical sheaf on \(P\), and let \(\mathscr{L}\) be
  a sufficiently ample ample invertible sheaf such that \(\mathscr{L} \otimes \omega_P\)
  remains very ample. Then a general smooth section of \(\mathscr{L}\)
  has no non-trivial automorphisms.
\end{theorem}

\begin{proof}
  The proof is to applying Theorem~\ref{theorem:mere}. Let \(X\) be the zero 
  locus of any smooth transverse section of \(\mathscr{L}\). By the 
  adjunction formula and our assumption, the canonical divisor 
  \(\omega_X\cong \mathscr{L}\otimes \omega_P\) is very ample. 

  We denote by \(\mathrm{H}^m(X; \mathbb{C})_p\) the kernel of the 
  cup-product operator
  \[
  \cup \omega_X\colon \mathrm{H}^m(X; \mathbb{C})\to \mathrm{H}^{m+2}(X; \mathbb{C})
  \]
  with \(m=\dim X\). Let \(Q\) be the intersection form. 
  The Hodge-Riemann bilinear relation shows that the Hermitian form
  \(h:=\oplus_{p+q=m} i^{p-q}Q(\cdot,\bar{\cdot})\) 
  is positive definite on 
  \[
  \mathrm{H}^m(X; \mathbb{C})_p=\bigoplus_{p+q=m} \mathrm{H}^{p,q}\cap \mathrm{H}^m(X; \mathbb{C})_p.
  \]
  Any automorphism of \(X\) preserves \(\omega_X\) and \(h\). By Lemma~\ref{lemma:faithful-general-type}, 
  \(\operatorname{Aut}(X)\) can be regarded as 
  a subgroup of the unitary group 
  \(\mathrm{U}(\mathrm{H}^m(X; \mathbb{C})_p, h)\).
  Additionally, \(\operatorname{Aut}(X)\) acts on the integral cohomology 
  \(\mathrm{H}^n(X;\mathbb{Z})\). Therefore, \(\Aut(X)\) is a discrete
  subgroup in a compact group, which is necessarily finite.

  Let \(B\) be the open subset of smooth divisors in the linear system 
  \(|\mathscr{L}|\).  Let \(\mathcal{X} \to B\) be the universal family
  of the smooth divisors. The above discussion implies that the relative 
  automorphism group scheme
  \begin{equation*}
    \mathrm{Aut}_B(\mathcal{X}) \to B
  \end{equation*}
  is a quasi-finite morphism. The hypersurface \(X\) is not uniruled
  since \(\omega_X\) is ample. Matsusaka-Mumford's theorem
  \cite{Matsusaka-Mumford:two-fundamental-theorems} ensures that the 
  morphism \(\mathrm{Aut}_B(\mathcal{X}) \to B\) is proper and hence 
  finite. It is unramifed since a group scheme over \(\mathbb{C}\) 
  is always smooth. Thus, Hypothesis~\ref{theorem:mere}(\ref{item:finiteness}) 
  is verified with the polarization \(\mathcal{L}\) being the relative 
  canonical bundle \(\omega_{\mathcal{X}/B}\)

  By comparison theorem, the \(\ell\)-adic cohomology of a complex variety
  is isomorphic to the Betti cohomology of the associated analytic space. 
  In the following, when we apply Theorem~\ref{theorem:mere}, we will use 
  Betti cohomology in place of the \(\ell\)-adic cohomology. Then 
  Lemma~\ref{lemma:faithful-general-type} affirms 
  Hypothesis~\ref{theorem:mere}(\ref{item:generic-faithfulness})

  To study the monodromy group, we shall consider a Lefschetz pencil \(D\) 
  in the linear system \(|\mathscr{L}|\) such that \([X]\in D\), 
  and the vanishing cohomology 
  \[
    \mathrm{H}_{ev}^m(X, \mathbb{Q}) \coloneqq \mathrm{Ker}(\mathrm{H}^m(X, \mathbb{Q})
    \overset{i_*}{\longrightarrow} \mathrm{H}^{m+2}(P, \mathbb{Q}))
  \]
  where \(i_*\) is the Gysin map. In general \(\mathrm{H}_{ev}^m(X, \mathbb{Q})\)
  is a subspace of \(\mathrm{H}^m(X,\mathbb{Q})_{\mathrm{prim}}\). 
  In this proof, we instead use the vanishing cohomology and consider 
  the monodromy group of the representation
  \begin{equation}
  \label{eqs:monod-van-cohom}
    \rho \colon \pi_1(D-S) \to \operatorname{Aut}(\mathrm{H}_{ev}^m(X;\mathbb{Q}), \psi)
  \end{equation}
  where \(S\subset D\) is the finite subset of critical values of the pencil. 
  
  Let \(\mathrm{H}\) denote \(\mathrm{H}_{ev}^m(X;\mathbb{Q})\). 
  If \(m\) is odd, Kazhdan-Margulis theorem~\cite[Théorème~5.10]{Deligne:Weil-I}
  shows that the geometric monodromy group of \(\rho\)~\eqref{eqs:monod-van-cohom}
  is the sympelctic group \(\mathrm{Sp}(\mathrm{H}, \psi)\).
  If \(m\) is even, the geometric monodromy group is either a 
  finite group or the orthogonal group \(\mathrm{O}(\mathrm{H}, \psi)\), 
  see~\cite[Théorème~4.4.2]{Deligne:Weil-II}. The case of finite group 
  occurs when the intersection form \(\psi\) is definite. By the Hodge
  index theorem, it deduces that \(\mathrm{H}^{p,q}(X)=0\) unless 
  \(p=q\). However, in our case \(\mathrm{H}^{m,0}(X)\) is non-trivial 
  since the canonical divisor \(\omega_X\) is very ample. 

  Let \(g\in \Aut(X)\) be any automorphism on \(X\). If the action 
  \(g^*\) preserves the subspace \(\mathrm{H}\), the same treatment
  for the proof of Theorem~\ref{theorem:mere} deduces that \(g^*\) 
  acts as a scalar on \(\mathrm{H}\). In the rest of the proof, let 
  us show that \(H\) is \(\Aut(X)\)-stable.
  
  By the Hard Lefschetz theorem, there is an orthogonal decomposition 
  \[
  \mathrm{H}^m(X, \mathbb{Q})=\mathrm{H}_{ev}^m(X, \mathbb{Q})\oplus j^*\mathrm{H}^m(P, \mathbb{Q})
  \]
  with respect to the intersection form \(\psi\), where 
  \(j\colon X\hookrightarrow P\) is the inclusion map. The direct sum 
  is also a decomposition of the monodromy representation, for which 
  the monodromy action on \(j^*\mathrm{H}^m(P, \mathbb{Q})\) is trivial.
  Then the Zariski closure of the image of the representation
  \begin{equation}
  \label{eqs:whole-monod-rep}
    \varphi \colon \pi_1(D-S) \to \operatorname{Aut}(\mathrm{H}^m
      (X;\mathbb{Q}), \psi)
  \end{equation}
  is the connected subgroup \(\mathrm{Sp}(\mathrm{H})\times E\) 
  (resp. \(\mathrm{O}(\mathrm{H})\times E\)) in \(\mathrm{GL}(\mathrm{H}^m(X, \mathbb{Q}))\)
  if \(m\) is odd (resp. even), where \(E\) denotes the identity 
  matrix of the subspace \(j^*\mathrm{H}^m(P, \mathbb{Q})\).
  By the similar finite \'etale base change argument in the proof of 
  Theorem~\ref{theorem:mere}, any automorphism action \(g^*\) on 
  \(\mathrm{H}^m(X, \mathbb{Q})\) commutes with the monodromy actions
  ~\eqref{eqs:whole-monod-rep}, i.e., \(g^*\) is an isomorphism of 
  the representation \(\varphi\). Consider the composition map
  \[
  pr_2\circ g^*|_{\mathrm{H}}: \mathrm{H}\to j^*\mathrm{H}^m(P, \mathbb{Q})
  \]
  where \(pr_2\) is the natural projection. It is easy to verify that
  \(pr_2\circ g^*|_{\mathrm{H}}\) is also a morphism of the representations
  ~\eqref{eqs:whole-monod-rep}. However, \(\mathrm{H}\) is an irreducible
  \(\pi_1(D\setminus S)\)-module, and \(j^*\mathrm{H}^m(P, \mathbb{Q})\) is 
  the trivial representation. It follows that \(pr_2\circ g^*=0\). 
  In conclusion, the action \(g^*\) on \(\mathrm{H}\) commutes with the monodromy group 
  \(\mathrm{Sp}(\mathrm{H})\) or \(\mathrm{SO}(\mathrm{H}, \psi)\), 
  thus \(g^*\) acts as a scalar on \(\mathrm{H}\).
  
  Finally we show that \(g\) is in fact the identity map on \(X\).
  Set \(\mathrm{H}^0_{ev}(X, \omega_X)\coloneqq \mathrm{H}^0(X, \omega_X)\cap \mathrm{H}^m_{ev}(X, \mathbb{C})\).
  Let \(U=P-X\) be the open complement. There is the exact sequence 
  \[
  \to \mathrm{H}^{m+1}(U, \mathbb{C}) \overset{res}{\to} \mathrm{H}^m(X, \mathbb{C})\overset{i_*}{\to} 
  \mathrm{H}^{m+2}(P, \mathbb{C})\to
  \]
  where \textit{res} is induced by the residue map. It is also an exact sequence of mixed Hodge 
  structures. Restricting on the Hodge filtrations yields the 
  following 
  \[
   F^{m+1}\mathrm{H}^{m+1}(U, \mathbb{C}) \overset{res}{\to} 
   F^m\mathrm{H}^m(X, \mathbb{C})\overset{i_*}{\to} F^{m+1}\mathrm{H}^{m+2}(P, \mathbb{C}) .
  \]
  Note that \(F^m\mathrm{H}^m(X, \mathbb{C})=\mathrm{H}^0(X, \omega_X)\). Hence 
  \(\mathrm{H}^0_{ev}(X, \omega_X)\) is the kernel of the homomorphism \(i_*\). 
  Since \(X\) is sufficiently ample in \(P\), It follows from~\cite[Thm. 6.5]{Voi-Hodge-II}
  that the Hodge filtration \(F^{m+1}\mathrm{H}^{m+1}(U, \mathbb{C})\) 
  is the image of
  \[
    \mathrm{H}^0(P, \omega_P(X))\to \mathrm{H}^{m+1}(U, \mathbb{C}),
  \]
  where \(\omega_P(X)=\omega_P\otimes \mathscr{L}\) is the sheaf of 
  meromorphic forms of degree \(m+1\) with simple poles along \(X\).
  Thus \(\mathrm{H}^0_{ev}(X, \omega_X)\) is identified with 
  \[
    \mathrm{Im}(\mathrm{H}^0(P, \omega_P(X))\to \mathrm{H}^0(X, \omega_X)).
  \]
  This description of \(\mathrm{H}^0_{ev}(X, \omega_X)\) yields the following 
  diagram
  \[
    \begin{tikzcd}
      X \ar[rr, hook] \ar[d, hook] && P \ar[d, hook] \\
      \mathbb{P}(\mathrm{H}^0(X, \omega_X)^{\vee}) \ar[r, dashrightarrow] & 
      \mathbb{P}(\mathrm{H}^0_{ev}(X, \omega_X)^{\vee}) \ar[r, hook] & \mathbb{P}(\mathrm{H}^0(P, \omega_P(X))^{\vee}).
    \end{tikzcd}
  \]
  The right vertical map is an embedding since \(\omega_P(X)\) is very ample. 
  Then the composition \(X\to \mathbb{P}(\mathrm{H}^0_{ev}(X, \omega_X)^{\vee})\) 
  is an \(\Aut(X)\)-equivariant embedding. Therefore \(g\) must be 
  the identity map on \(X\) since the scalar matrix \(g^*\) acts identically on 
  \(\mathbb{P}(\mathrm{H}^0_{ev}(X, \omega_X)^{\vee})\).
  \end{proof}

  The original proof of Theorem~\ref{theorem:general-type-hypersurface} 
  is incomplete. The argument presented is communicated with H.-Y.~Lin and E.~Shinder.

\section{Automorphism of complete intersection in projective spaces.}
\label{sec:auto-ci}

Let \(k\) be an algebraically closed field. Let \((d_1,\ldots,d_c;n)\)
be a sequence of positive integers with \(2\leq d_i\leq d_{i+1}\), \(n>c\).
We say a complete intersection \(X\subset \mathbb{P}^n_k\) is of type
\((d_1,\ldots,d_c;n)\) if the ideal of \(X\) is generated by \(c\) 
homogeneous polynomials with the prescribed degrees \(\{d_1, \ldots, d_c\}\). 
The goal of this section is proving the following:

\begin{theorem}
\label{theorem:gen-trivial-auto}
Let \(k\) be an algebraically closed field of characteristic
\(p \neq 2\). Suppose that \((d_{1},\ldots,d_{c})\) is a collection 
of positive integers satisfying one of the following conditions
\begin{itemize}
\item \(3 \leq d_{1} \leq d_{2} \leq \cdots \leq d_{c}\) and \(p \nmid d_{1}\);
\item \(c \geq 3\) and \(2 = d_{1} = d_{2} = d_{3} \leq \cdots \leq d_{c}\).
\end{itemize}
Then a general smooth complete intersection \(X\subset \mathbb{P}^{n}_{k}\) 
of multidegree \((d_{1},\ldots,d_{c})\) with \(c < n\) has no-nontrivial 
linear automorphisms unless \(X\) is a cubic curve.
\end{theorem}


\begin{observation}
\label{observation:reduction}
Let \(X\) be a general complete intersection of multidegree
\((d_1, \ldots, d_c)\), and \(F_{1},\ldots,F_{c}\) denote the defining
polynomials of \(X\) with \(\deg F_i=d_i\). Suppose that \(r\) is the
maximal number such that \(d_1=\cdots=d_r<d_{r+1}\). Then the action 
of any linear automorphism of \(X\) preserves the ideal 
\(( F_1,\ldots, F_r)\). Regard \(X\) as a subscheme of the complete 
intersection \(Y\) defined by \((F_1,\ldots, F_r)\), then we have
\(\Aut_L(X)\subset \Aut_L(Y)\). To prove \(\Aut_L(X) = \{1\}\), 
it suffices to prove a general complete intersection with equal 
multidegrees has no non-trivial automorphisms.
\end{observation}

\begin{remark}
   \label{remark:ci-quadric}
   The complete intersections of type \((d_1, \ldots, d_c)\) with 
   \(2 = d_1 \leq d_2 < d_3 \leq \cdots \leq d_c\) are not included
   in Theorem~\ref{theorem:gen-trivial-auto}. According to Observation
   ~\ref{observation:reduction}, the corresponding \(Y\) are 
   hyperquadrics or complete intersections of two quadrics so that \(Y\)
   admit non-trivial linear automorphisms, see~\cite{Reid-ci-quadric}. 
   To treat these types, one may find a way to characterize the 
   autmorphism group for complete intersections of type 
   \((2,d,\ldots,d;n)\) and \((2,2, d,\ldots,d;n)\) for \(d > 2\).
\end{remark}

To apply Theorem~\ref{theorem:mere} to complete intersections with
equal multidegrees, we verify Hypothesis
~\ref{theorem:mere}\eqref{item:generic-faithfulness} for a particular
complete intersection of Fermat type. We will characterize its 
automorphism group. Using the work by Aoki~\cite{Aoki:ci-Fermat-type} 
on the decomposition of its cohomology group , we can prove that the 
automorphism group acts faithfully on the cohomology group.

\subsection*{Automorphism of complete intersections of Fermat type}
    \label{subsec:Auto-ci-Fermat}
Let \(k\) be an algebraically closed field. Fix natural numbers
\(n\geq 3, r\geq 2\), and \(d\geq 2\), where \(d\) is invertible in 
\(k\). Let \(X:=X_{n,r,d}\subset \mathbb{P}^{n}_k\) be the complete 
intersection defined by the Fermat-like equations:
\begin{equation}
    \label{equation:Fermat-eqs}
    \left\{
        \begin{matrix}
        x_0^d+\cdots +x_n^d=0,\\

        \lambda_0 x_0^d+\cdots +\lambda_n x_n^d=0,\\

        \vdots\\

        \lambda_0^{r-1} x_0^d+\cdots+\lambda_n^{r-1} x_n^d=0,
        \end{matrix}
    \right.
\end{equation}
where \(\lambda_i\neq \lambda_j\in k^*, \forall i, j\). It is known 
that \(X\) is non-singular~\cite[Prop. 2.4.1]{Terasoma:fermat-complete-intersection}.
Let \(\mu_d\) be the group of $d$-th roots of unity in \(k\),
and $\mu_d^{n+1}$ be the $(n+1)$-th product of $\mu_d$.
Denote by $G_n^d$ the quotient group $\mu_d^{n+1}/\Delta(\mu_d)$,
where $\Delta: \mu_d\to \mu_d^{n+1}$ is the diagonal embedding.
The group $G_n^d$ naturally acts on the variety $X$:
\[
[\xi_0,\ldots, \xi_n]\cdot (x_0:\ldots: x_n)\mapsto (\xi_0 x_0:\ldots: \xi_n x_n), ~\xi_i\in \mu_d
\]

\begin{proposition}
\label{proposition:Aut-Fermat-ci}
Let $X:=X_{n,r,d}\subset \mathbb{P}^n_k$ be the Fermat complete 
intersection defined by the equations above. Denote by $\Aut_L(X)$ 
the group of linear automorphisms of \(X\). There is a group 
homomorphism 
\[
\sigma\colon \Aut_L(X)\to \mathfrak{S}_{n+1},
\]
where $\mathfrak{S}_{n+1}$ is the permutation group of $n+1$ elements.
such that the kernel of \(\sigma\) is the group \(G^d_n\). Moreover, 
if \((\lambda_{0},\ldots,\lambda_{n})\in k^{n+1}\) is a general point,
then \(\sigma\) is the trivial map, i.e., \(\Aut_L(X)=G^d_n\).
\end{proposition}

\begin{proof}
Let \((x_0\colon \dots \colon x_n)\) be a coordinate in \(\mathbb{P}^n_k\).
The action of a linear automorphism \(g\in \Aut_L(X)\) transforms 
each \(x_i\) to a linear form
\[
g_i:=g^*(x_i)=\sum_{j=0}^{n} a_{ij}x_j, ~a_{ij}\in k.
\]
Since the ideal of $X$ is generated by the
Fermat-like equations~\eqref{equation:Fermat-eqs}, we observe
\begin{equation}
\label{equation:linear-transform}
g^*(x_0^d+\cdots +x_n^d)=\sum_{i=0}^{n} g_i^d=\sum_{i=0}^{n} c_ix_i^d,
\text{~for some~} c_i\in k.
\end{equation}
Differentiating both sides of \eqref{equation:linear-transform}
with respect to \(\frac{\partial}{\partial x_j}\), we obtain
    \[
    d\sum_{i=0}^{n} g^{d-1}_i \cdot \frac{\partial g_i}{\partial x_j}
    = d\sum_{i=0}^{n} g^{d-1}_i a_{ij}= d\cdot c_j \cdot x_j^{d-1}.
    \]
Focusing on the coefficients of the monomial $x_0^{d-1}$ on both sides. 
There is the relation
    \[
        d\cdot \sum_{i=0}^{n} a_{ij}a_{i0}^{d-1}=
        \begin{cases}
            0,          & j\neq 0;\\
            d\cdot c_0, & j=0.
        \end{cases}
    \]
Let $A$ represent the matrix $(a_{ij})$. Since \(d\) is invertible in 
the field \(k\), we have
    \[
    (a_{00}^{d-1},\ldots, a_{n0}^{d-1})\cdot A=(c_0,0,\ldots,0).
    \]
Similarly,
    \[
    g^*(\lambda_0 x_0^d+\cdots +\lambda_n x_n^d)=\sum_{i=0}^n \lambda_i g_i^d
    =\sum_{i=0}^{n} c_i'x_i^d, \text{~for some~} c'_i\in k.
    \]
Following the same discussion as above, we have
\[
(\lambda_{0} a_{00}^{d-1},\ldots, \lambda_n a_{n0}^{d-1})\cdot A=(c_0',0, \ldots, 0).
\]
Note that the matrix $A$ is invertible, and \(c_0, c'_0\) are non-zero.
Hence \((a_{00}^{d-1},\ldots, a_{n0}^{d-1})\) is proportional 
to \((\lambda_{0} a_{00}^{d-1},\ldots, \lambda_n a_{n0}^{d-1})\).
Since \(\lambda_{0},\lambda_{1},\ldots,\lambda_{n}\) are pairwise 
distinct, only one element among \(\{a_{00},\ldots, a_{n0}\}\) is 
non-zero. Considering the coefficients of $x_j^{d-1}$ for each $j$, 
the same conclusion holds for each tuple \((a_{0j},\ldots,a_{nj})\). 
Therefore, \(g\) determines a permutation \(\sigma_g\in \mathfrak{S}_{n+1}\) 
such that \(g^*(x_i)=a_{\sigma_g(i)i}x_{i}\). The assignment
\(g\mapsto \sigma_g\) thus establishes a group homomorphism
\[
\sigma: \Aut_L(X)\to \mathfrak{S}_{n+1}.
\]

In the following, we demonstrate that the kernel of $\sigma$ is $G^d_n$.
Suppose that $g\in \Ker \sigma$. Then $g$ is represented by a diagonal 
matrix $(a_{ii})_{0\leq i\leq n}$. The action of \(g\) on the 
Fermat polynomials \(\{\sum\limits_{i=0}^{n} \lambda_{i}^{j} x_{i}^{d}
~|~ 0\leq j\leq r-1\}\) yields an $r\times r$ matrix $B$ satisfying 
the relationship
\begin{equation}
\label{equation:diagonal-action}
B \cdot
\begin{pmatrix}
1  & \cdots & 1\\
\lambda_0 &\cdots & \lambda_n\\
\vdots & &\vdots\\
\lambda_0^{r-1} &\cdots & \lambda_n^{r-1}
\end{pmatrix}
=
\begin{pmatrix}
1  & \cdots & 1\\
\lambda_0 &\cdots & \lambda_n\\
\vdots & &\vdots\\
\lambda_0^{r-1} &\cdots & \lambda_n^{r-1}
\end{pmatrix}
\cdot
\begin{pmatrix}
a^d_{00} & &\\
& \ddots &\\
& & a^d_{nn}
\end{pmatrix}
\end{equation}
Let us view each column\((1,\lambda_i,\ldots, \lambda_i^{r-1})^{\mathsf{T}}\)
as a vector \(v_i\) in the \(r\)-dimensional vector space \(k^{r}\). 
Then \(B\) represents a linear map with eigenvectors \(\{v_0,\ldots, v_n\}\)
and eigenvalues \(\{a_{00}^d, \ldots, a_{nn}^d\}\). Note that any $r$ 
elements among \(\{v_0,\ldots,v_n\}\) form a basis of \(k^r\) since 
the basis corresponds to a non-degenerate Vandermonde matrix. 
It implies that \(a_{ii}^d=a_{jj}^d\) for all \(i,j\). Consequently, 
the matrix \((a_{ii})\) can be represented by
\[
\begin{pmatrix}
\xi_0\cdot a_{00} & &\\
& \ddots & \\
& & \xi_n\cdot a_{00}
\end{pmatrix}
,~\textrm{for some}~\xi_i\in \mu_d.
\]
Hence, the action of $g$ on $X$ is equivalent to the action of 
$[\xi_0,\cdots,\xi_n]\in G_n^d$.

Now let us prove the final assertion. Given a linear automorphism
\(g\), let \(\sigma\) denote the permutation
\(\sigma(g)\in \mathfrak{S}_{n+1}\), and \((a_{\sigma(i)i})\)
be the matrix representing \(g\). Again, the action of $g^*$ on the 
Fermat polynomials~\eqref{equation:Fermat-eqs} yields an $r\times r$ 
matrix $B:=(b_{ij})_{0\leq i,j\leq r-1}$ such that
\[
B \cdot
\begin{pmatrix}
1  & \cdots & 1\\
\lambda_0 &\cdots & \lambda_n\\
\vdots & &\vdots\\
\lambda_0^{r-1} &\cdots & \lambda_n^{r-1}
\end{pmatrix}
=
\begin{pmatrix}
1  & \cdots & 1\\
\lambda_0 &\cdots & \lambda_n\\
\vdots & &\vdots\\
\lambda_0^{r-1} &\cdots & \lambda_n^{r-1}
\end{pmatrix}
\cdot
(a_{\sigma(i)i}).
\]
By this relation, the first two rows of the matrix $B$ correspond 
to two polynomials
\[
q(t)=\sum_{j=0}^{r-1}b_{0j}t^j, ~p(t)=\sum_{j=0}^{r-1}b_{1j}t^j
\]
satisfying the interpolation data
\[
q(\lambda_i)=a_{\sigma(i)i}, ~p(\lambda_i)=\lambda_{\sigma(i)}a_{\sigma(i)i}, 
\textrm{~for~} 0\leq i\leq n.
\]

Consider the Lagrangian polynomial
\[
L(t):=\sum_{k=0}^{n} a_{\sigma(k)k} \prod_{0\leq j\leq n, j\neq k} 
\frac{t-\lambda_j}{\lambda_k-\lambda_j}
\] 
which interpolates the given data \((\lambda_i, \alpha_{\sigma(i)i})\).
Assume that the constants \(\{\alpha_{\sigma(i)i}\}_{0\leq i\leq n}\) 
are not all equal. It is known that, for a generic choice of the tuple 
\((\lambda_0, \ldots, \lambda_n)\) in \(k^{n+1}\), the degree of 
\(L(t)\) is \(n\). Hence \(L(t)=q(t)\) since \(L(t)-q(t)\) contains 
\(n+1\) distinct roots. However, we have \(\deg q(t) = r-1 < n\) 
as a contradiction. Therefore, all \(a_{\sigma(i)i}\) must be equal, 
and \(q(t)\) is the constant polynomial \(b_{00}\). It follows that 
the second interpolation data becomes
\[
p(\lambda_i)=b_{00}\lambda_{\sigma(i)}.
\]
Applying the Lemma \ref{interpolation-poly} below, such a polynomial 
$p(t)$ exists for a generic choice of $(\lambda_0, \ldots, \lambda_n)$ 
in \(k^{n+1}\) only if $\sigma=\operatorname{Id}$. Thus, the last 
assertion follows.
\end{proof}

\begin{lemma}\label{interpolation-poly}
Let \(\sigma\in \mathfrak{S}_{n+1}\) be a non-trivial permutation,
and let \(r\) be an integer with \(1<r<n\). Suppose that 
\((\lambda_0,\ldots, \lambda_n)\in k^{n+1}\) is a general point. 
Then there exists no polynomial function $p(t)\in k[t]$ with
$\deg p(t)\leq r-1$ that interpolates the $n+1$  data points
$(\lambda_0, \lambda_{\sigma(0)}),\ldots, (\lambda_n, \lambda_{\sigma(n)})$,
i.e.,
    \[
    p(\lambda_i)=\lambda_{\sigma(i)}, 0\leq i\leq n.
    \]
\end{lemma}

\begin{proof}
We identify the space of one-variable polynomials of degree at most 
$r-1$ with the $r$-dimensional vector space $k^r$. We consider the 
incidence subvariety
    \[
    \Gamma_{\sigma}:=\{(p(t), (y_0,\ldots,y_n))\in k^r\times k^{n+1} ~|~ p(y_i)-y_{\sigma(i)}=0, ~0\leq i\leq n\}.
    \]
Let $\pi_1:\Gamma_{\sigma}\to k^{r}$ and $\pi_2: \Gamma_{\sigma}\to k^{n+1}$
be natural projections. Suppose that \((p(t), (\lambda_0,\ldots, \lambda_n))\)
is a point in \(\Gamma_\sigma\), and \(N\) is the order of $\sigma$, 
then
    \[
    p^{\circ N}(\lambda_i)=\lambda_{\sigma^N(i)}=\lambda_i, 
    ~\forall~ 0\leq i\leq n.
    \]

\textbf{Case 1.} The polynoimal $p^{\circ N}(t)-t$ is non-zero. Then 
$p^{\circ N}(t)-t$ has finite roots. Hence the possiblity of 
\(\{\lambda_0,\ldots, \lambda_n\}\) are finitely many. It follows that 
\(\pi_1\) is a quasi-finite map, and the dimension of $\Gamma_\sigma$ 
is $r$. Then $\pi_2(\Gamma_\sigma)$ is a proper subset in $k^{n+1}$. 
Therefore, for a generic $(\lambda_0,\ldots, \lambda_n)\in k^{n+1}$, 
there exists no $p(t)\in k^r$ satisfying the condition 
$p(\lambda_i)=\lambda_{\sigma(i)}, 0\leq i\leq n$.

\textbf{Case 2.} The polynoimal $p^{\circ N}(t)-t$ is zero. This situation
occurs only if $p(t)$ is a linear form $\alpha+\beta t$ with \(\beta^N=1\). 
All such linear forms consitutes a one-dimensional subset 
\(\Xi\subset k^r\). We claim that $\dim \pi_1^{-1}(\Xi)<n+1$. The fiber 
of \(\pi_1\) over a linear form \(\alpha+\beta t\) is a subset in 
\(k^{n+1}\) defined by \(n+1\) equations
    \[
    \alpha+\beta y_i=y_{\sigma(i)}, 0\leq i\leq n.
    \]
Since \(n > r\geq 2\) and \(\sigma\neq \operatorname{Id}\), these equations
impose at least two constraints on \(k^{n+1}\). Therefore,
the dimension of \(\pi_1^{-1}(\Xi)\) is less than \(n+1\). In conclusion, 
for a generic choice of $(\lambda_0,\ldots, \lambda_n)\in k^{n+1}$, 
no linear form can interpolate the data 
$(\lambda_i, \lambda_{\sigma(i)})_{0\leq i\leq n}$.
\end{proof}

\subsection*{Cohomology of complete intersection of Fermat type}
    \label{subsec:coh-ci-Fermat}
Fix a prime \(\ell\) invertible in the field \(k\). Let \(\mathbb{Q}_{\ell}\)
be the \(\ell\)-adic numbers, and let \(K\) be a field extension of 
\(\mathbb{Q}_{\ell}\) containing \(d\)-th roots of unity. Let 
\(\mathrm{H}^i_{\mathrm{prim}}(X, K)\) denote the primitive part of 
the \'etale cohomology \(\mathrm{H}^i_{\mathrm{et}}(X, K)\), namely, 
the orthogonal complement of the image 
\(\mathrm{H}^i_{\mathrm{et}}(\mathbb{P}^n, K)\to \mathrm{H}^i_{\mathrm{et}}(X, K)\).
The character group of \(G_n^d\) is identified with the group
\[
   \hat{G}^d_n\coloneq\{\chi = (a_0,\ldots,a_n)\in (\mathbb{Z}/d\mathbb{Z})^{\oplus n+1}~|~\sum a_i=0\}.
\]
The  group action of \(G_n^d\) on the middle cohomology 
\(\mathrm{H}^{n-r}_{\mathrm{prim}}(X_{n,r,d}, K)\) induces the
decomposition
\begin{equation}
\label{eqs:eigen-decomp}
    \mathrm{H}^{n-r}_{\mathrm{prim}}(X, K)\cong \bigoplus_{\chi\in \hat{G}^d_n} V_{\chi},
\end{equation}
where \(V_{\chi}\) is the \(\chi\)-eigenspace 
\[
    \{v\in \mathrm{H}^{n-r}_{\mathrm{prim}}(X, K)~|~ g(v)=\chi(g)\cdot v, \forall g\in G^d_n\}.
\] 
For a smooth complete intersection of Fermat type defined over the 
complex numbers and the corresponding decomposition on the singular 
cohomology, the dimension of \(V_{\chi}\) has been characterized 
by Aoki~\cite[Thm. 1.2]{Aoki:ci-Fermat-type}. It is parallel to 
apply his approach to \'etale cohomology
\begin{proposition}
\label{prop:cohom-decompo-ci}
Let \(X\coloneq X_{n,r,d}\) be the complete intersection of Fermat 
type defined in~\eqref{equation:Fermat-eqs}. For a character 
\(\chi=(a_0,\ldots,a_n)\in \hat{G}^d_n\), set \(I(\chi)\) to be the 
subset  \(\{i\in \{0, \ldots, n\}~|~a_i=0\}\) and \(\nu:=\#I(\chi)\). 
With the decomposition~\eqref{eqs:eigen-decomp} we have
\[
    \dim V_{\chi}=\binom{n-\nu-1}{r-\nu-1}, ~\forall~ 0\leq \nu\leq r-1.
\]
\end{proposition}
\begin{proof}
Let \(\pi\colon \mathbb{P}_k^n\to \mathbb{P}_k^n\) be the morphism 
\(\pi((x_0,\ldots, x_n))=(x_0^d,\ldots, x_n^d)\). Let \(L\) be the
subspace of \(\mathbb{P}^n\) defined by the linear equations 
\[
    \sum_{j=0}^{n}\lambda^i_j x_j, 0\leq i\leq r-1.
\]
Then \(X=\pi^{-1}(L)\). Since the morphism \(\pi\) is finite, we have
\(\mathrm{H}^i_{\et}(X, K)\cong \mathrm{H}^i_{\et}(L, \pi_*K)\). Since 
the characteristic of the coefficient field \(K\) is zero, and \(K\) 
contains \(d\)-th roots of unity, we have 
\(\pi_*K=\bigoplus_{\chi} (\pi_*K)_\chi\), and 
\[
    \mathrm{H}^i_{\et}(X, K)_\chi \cong \mathrm{H}^i_{\et}(L, (\pi_*K)_\chi).
\]
By the Lefschetz theorem, for \(i\neq n-r\), we have
\(\mathrm{H}^i_{\et}(X, K) \cong \mathrm{H}^i(\mathbb{P}^n, K)\).
Hence, for \(\chi\neq 0\), it can deduce that
\[
\chi(L, (\pi_*K)_\chi)=(-1)^{n-r}\dim V_\chi.
\]
The remaing proof of showing 
\(\chi(L, (\pi_*K)_\chi)=(-1)^{n-r}\binom{n-\nu-1}{r-\nu-1}\) follows 
the same steps as in~\cite[Thm. 1.2]{Aoki:ci-Fermat-type}.
\end{proof}

\begin{proposition}
    \label{proposition:auto-faithful-fermat-ci}
    Let $X:=X_{n,r,d}\subset \mathbb{P}^n_k$ be the complete intersection
    of Fermat type defined in~\eqref{equation:Fermat-eqs}.
    Fix a prime \(\ell\neq \mathrm{char}(k)\).
    Then the group $\Aut_L(X)$ of linear automorphisms acts faithfully
    on the primitive cohomology $\mathrm{H}^{n-r}_{\mathrm{prim}}(X, \mathbb{Q}_{\ell})$
    unless \(d=r=2\) and \(n+1\) is even.
\end{proposition}
\begin{proof}
We first prove the action of the subgroup \(G^d_n\subset \mathrm{Aut}_L(X)\) 
on the cohomology is faithful. Suppose that \(g\in G^d_n\) acts trivially 
on \(\mathrm{H}^{n-r}_{\mathrm{prim}}(X, \mathbb{Q}_{\ell})\). On the 
eigenspace decomposition~\eqref{eqs:eigen-decomp}, the field extension 
\(\mathbb{Q}_{\ell}\subset K\) has no impact on the action of \(g\). 
It follows that \(\chi(g)=1\) if \(V_\chi\neq \emptyset\). To prove
the faithfulness, we need find a set \(S\) of characters of \(G^d_n\) 
such that
\begin{itemize}
    \item \(V_{\chi}\neq 0, ~\forall \chi\in S\);
    \item \(\{\sigma\in G^d_n~|~ \chi(\sigma)=1, \forall \chi\in S\}=\{\operatorname{Id}\}.\)
\end{itemize}

Choose integers $k,s,t\in \mathbb{Z}_{\geq 0}$ satisfying
\[
d \mid k+n,~ (n+t, d)=1 \text{~and~} s+t\equiv n-1\ \textrm{mod}\ d. 
\]
Consider the following characters
    \begin{align*}
    \chi_{k}&:=(1,\ldots, 1, k);\\
    \chi_{(s,t,i)}&:=(1,\ldots, s,\ldots,t), 1\leq i\leq n,
    \end{align*}
where \(\chi_{(s,t,i)}\) means the \(i\)-th term is \(s\), the 
\((n+1)\)-th term is \(t\), and \(1\) otherwise. For the character 
\(\chi_k\) the number \(\nu =\#I(\chi_k)\leq 1\). 
By Proposition~\ref{prop:cohom-decompo-ci} we have
\[
\dim V_{\chi_k}=\binom{n-\nu-1}{r-\nu-1} >0.
\] 
For the eigenspace \(V_{\chi_{(s,t,i)}}\), we look into the following
two situations 
\begin{enumerate}
    \item \(r\geq 3\). Since \(\nu=\#I(\chi_{(s,t,i)})\leq 2\) we have
    \(\dim V_{\chi_{(s,t,i)}}=\binom{n-\nu-1}{r-\nu-1}>0\).
    \item \(d\geq 3\). As shown in Lemma~\ref{lem:basis-characters},
    one can choose integers \(t\) and \(s\) such that at least one of 
    them is non-zero. Hence \(\nu=\#I(\chi_{(s,t,i)})\leq 1\) and 
    \(\dim V_{\chi_{(s,t,i)}}=\binom{n-\nu-1}{r-\nu-1}>0\).
\end{enumerate}
Therefore 
\[
    \chi_k(g)=\chi_{(s,t,i)}(g)=1,~1\leq i\leq n
\]
holds for the cases \(d\geq 3, r\geq 2\) and \(d\geq 2, r\geq 3\). 
It follows from Lemma~\ref{lem:basis-characters} that \(g\) is the 
identity element.

The remaining case \(d=2, r=2\), the complete intersection
of two quadrics, is well-known~\cite{Reid-ci-quadric}.
For the sake of completeness, we use the above argument to demonstrate it. 
\begin{enumerate}
    \item \(n+1\) is odd. Set \(\{\chi_i=(1,\ldots, a_i, \ldots,1)~|~0\leq i\leq n~,a_i=0\}\). 
    Then \(\nu=\#I(\chi_i)=1\) and \(\dim V_{\chi_{i}}=\binom{n-2}{r-2}=1\). 
    Write \(g=(b_0,\ldots,b_n)\in G^2_n\). It is easy to check that 
    \(\chi_i(g)=1\) for all \(0\leq i\leq n\) implies 
    \(b_j\equiv b_j'\ \textrm{mod}\ 2\), so that \(g=\mathrm{Id}\in G^2_n\). 
    Indeed, note that 
    \(\{\chi_i~|~0\leq i\leq n\}=\{\chi_k, \chi_{(s,t,i)}~|~0\leq i\leq n\}\) 
    for \(k=0, s=0, t=1\).
    \item \(n+1\) is even. We see from Proposition~\ref{prop:cohom-decompo-ci}
    that in this case \(V_{\chi}\) is non-zero if and only if 
    \(\chi = (1, \ldots, 1)\). Hence \(\mathrm{H}^{n-2}_{\mathrm{prim}}(X, K)=V_{\chi}\). Let 
    \(e_{ij}\in G^2_n\) be the transformation
    \[
    (x_0:\cdots: x_i:\cdots: x_j: \cdots : x_n)\mapsto (x_0:\cdots: -x_i:\cdots: -x_j: \cdots : x_n).
    \]
    Then \(e_{ij}^*\) acts trivially on \(V_{\chi}\),
    but \(e_{ij}\) is not the identity map.
\end{enumerate}  

In conclusion, the action of $G^d_n$ on $\mathrm{H}^{n-r}_{\mathrm{prim}}(X, K)$,
as well as on \(\mathrm{H}^{n-r}_{\mathrm{prim}}(X, \mathbb{Q}_{\ell})\), 
is faithful except for \(X\) is an odd-dimensional complete 
intersection of two quadrics. Hence we can already prove \(\Aut_L(X)\) acts 
faithfully on the primitive cohomology for a generic complete intersection \(X\) 
of Fermat type since \(\Aut_L(X)=G^d_n\) by Proposition~\ref{proposition:Aut-Fermat-ci}. 

Now consider the case that \(G^d_n < \Aut_L(X)\) is a proper subgroup. 
Given a character \(\chi\in \hat{G}^d_n\) and an 
automorphism $\sigma\in \Aut_L(X)$, we define the character 
$\chi^{\sigma}$ to be
\[
\chi^{\sigma}(g):=\chi(\sigma g\sigma^{-1}).
\]
Note that \(\sigma g\sigma^{-1}\in G^d_n\) since \(G^d_n\) is a normal
subgroup by Proposition~\ref{proposition:Aut-Fermat-ci}. The action 
$\sigma^*$ transfers a $\chi$-eigenspace to a $\chi^{\sigma}$-eigenspace 
if the $\chi$-eigenspace is nontrivial. In fact, for any 
$v\in \mathrm{H}^{n-r}_\mathrm{prim}(X, K)_{\chi}$
    \[
    g^*\sigma^*(v)=\sigma^*(\sigma g\sigma^{-1})^*(v)
    =\sigma^*(\chi(\sigma g \sigma^{-1})v)=\chi^{\sigma}(g)\sigma^*(v).
    \]
Suppose that $\sigma^*$ acts trivially on 
$\mathrm{H}^{n-r}_\mathrm{prim}(X, K)$. Then \(\chi=\chi^{\sigma}\)
for any \(\chi\in S\). Hence
    \[
    \chi(g)=\chi(\sigma g\sigma^{-1}), ~\forall g\in G^d_n.
    \]
From the preceding discussion, it follows that $g=\sigma g\sigma^{-1}$
for all $g\in G^d_n$. Let $g$ be a diagonal matrix with distinct entries. 
The relation $\sigma g=g\sigma$ implies $\sigma$ must be a diagonal matrix, 
thus $\sigma$ is contained in $G^d_n$. Given that the action of $G^d_n$
on the cohomology is faithful, we conclude that $\sigma$ is the identity map.
\end{proof}

\begin{lemma}\label{lem:basis-characters}
Let \(d\) and \(n\) be positive integers. Choose integers
\(k, t, s\in \mathbb{Z}\) satisfying
\begin{equation}
\label{eqs:arithmetic-cond}
d \mid k+n,~ (n+t,d)=1 \textrm{~and~} s+t\equiv k+1\ \textrm{mod}\ d.
\end{equation}
Consider the $(n+1)\times (n+1)$-matrix
\begin{equation}
A:=\begin{pmatrix}
s &\ldots & 1 & t\\
\vdots &  \ddots & \vdots &\vdots \\
1 &\ldots & s & t\\
1 &\ldots & 1 & k
\end{pmatrix}\in M_{n+1}(\mathbb{Z}/d\mathbb{Z}),
\end{equation}
for the tuple $(1,\ldots,s,\ldots,t)$ we place $s$ in the $i$-th term, 
and place $t$ in the $n+1$-term, and $1$ otherwise. Then the set of solutions 
\(\{x\in (\mathbb{Z}/d\mathbb{Z})^{\oplus n+1} ~|~A\cdot x=0\}\)
consists of the diagonal elements $\{(a,\cdots,a) ~|~ a\in \mathbb{Z}/d\mathbb{Z}\}$.
Moreover, when \(d\geq 3\), one can choose integers \(t\)
and \(s\) such that at least one is non-zero.
\end{lemma}
\begin{proof}
Suppose that \(x=(b_1,\ldots, b_{n+1})\) is a solution of 
\(A\cdot x=0\). Then we have
\[
\sum_{j\neq i, 1\leq j\leq n}b_j+sb_i+tb_{n+1}\equiv 0\ \textrm{mod}\ d, 
~\sum_{j=1}^{n}b_j+kb_{n+1}\equiv 0\ \textrm{mod}\ d.
\]
It follows that
    \begin{equation}
        \label{solution-d-mod}
        (s-1)b_i+(t-k)b_{n+1}\equiv 0\ \textrm{mod}\ d, ~\forall 1\leq i\leq n.
    \end{equation}
Given the assumption on the integers $k,t,s$, we have
\[
(s-1)\equiv (k-t)\ \textrm{mod}\ d, ~(t-k)\equiv (n+t)\ \textrm{mod}\ d.
\]
Since $n+t$ is coprime to $d$, it follows that $k-t$ and \(s-1\) are 
invertible in $\mathbb{Z}/d\mathbb{Z}$. Then the equation~\eqref{solution-d-mod}
implies
\[
b_i\equiv b_{n+1}\ \textrm{mod}\ d,~ \forall {1\leq i\leq n}.
\]

Now let \(d\geq 3\). We prove the last assertion by contradiction.
Suppose that only \(t=s=0\) satisfies the arithmetic condition~\eqref{eqs:arithmetic-cond}
From this we obtain the relation
\[
    k\equiv -1\ \textrm{mod}\ d, ~n\equiv 1\ \textrm{mod}\ d.
\] 
Since \(d\geq 3\) there always exists an invertible element 
\(t'\in (\mathbb{Z}/d\mathbb{Z})^*\) such that \(t'\neq 1\). Taking
\(t=t'-1\neq 0\) yields \(n+t\equiv t'\ \textrm{mod}\ d\). Since \(t'\)
is invertible we can assert \((n+t, d)=1\). It leads to a contradiction.
\end{proof}

Now let us prove Theorem~\ref{theorem:gen-trivial-auto}.
\begin{proof}[Proof of Theorem~\ref{theorem:gen-trivial-auto}]
    By Observation~\ref{observation:reduction}, the complete
    intersection \(X\subset \mathbb{P}^n\) of type \((d_1, \ldots, d_c)\),
    under the assumptions in Theorem~\ref{theorem:gen-trivial-auto}, 
    can be regarded as a subscheme of the complete intersection 
    defined by the first \(r\) polynomials with minimal degree 
    \(d:=d_1\) where \(d\geq 3\), \(r\geq 1\), or \(d=2\), \(r\geq 3\). 
    Thus we shall prove that a general codimension \(r\) complete 
    intersection \(Y\subset \mathbb{P}^n\) of hypersurfaces with 
    equal degree \(d\) has no non-trivial linear automorphisms.

    \text{(i)}
    Suppose that \(d\geq 3\), \(r=1\). This reduces to the cases of
    hypersurfaces. A general smooth hypersurface \(Y\subset \mathbb{P}^n\) 
    of \(\deg Y\geq 3\) has no linear automorphisms unless \((d;n) = (3;2)\),
    see~\cite{Matsumura-Monsky:hypersurfaces-auto,Poonen:auto-hypersurface}.

    \text{(ii)}
    Suppose that \(d\geq 3\), \(r\geq 2\) or \(d=2\), \(r\geq 3\).
    Let us verify Hypotheses~\eqref{item:finiteness}--\eqref{item:big-monodromy}
    in Theorem~\ref{theorem:mere}.
    
    \begin{enumerate}
        \item Set \(\pi: \mathcal{X}\to B\) as the universal family
            of smooth complete intersections of type \((d_1,\ldots, d_c;n)\).
            The relative automorphism schemes \(\Aut_B(\mathcal{X})\to B\) 
            is a finite group scheme over \(B\) if the moduli stack 
            of smooth complete intersections is a separated and 
            Delgine-Mumford stack, see~\cite[Lem. 7.7]{Laumon-Moret-Bailly:champs}
            or \cite[Lem. 2.3]{Javanpeykar-Loughran:moduli-ci}. 
            Benoist affirmed that the moduli stack is separated and 
            Delgine-Mumford if \((d_1,\ldots, d_c;n)\neq (2;n)\), 
            see~\cite{Benoist:separated-moduli-ci}. Hence 
            Hypothsis~(\ref{item:finiteness}) holds.

            \item Let \(X_{n,r,d}\subset \mathbb{P}^n\) be the complete intersection 
            of Fermat type defined as~\eqref{equation:Fermat-eqs}
            with \(d\geq 3, r\geq 2\) or \(d=2, r\geq 3\).
            By Proposition~\ref{proposition:auto-faithful-fermat-ci}
            \(\Aut_L(X_{n,r,d})\) acts faithfully on 
            \(\mathrm{H}^{n-r}_{\mathrm{prim}}(X_{n,r,d})\). Thus
            the action is also faithful on \(\mathrm{H}^{n-r}(X_{n,r,d})\).  
            Then Hypothesis~(\ref{item:generic-faithfulness}) is satisfied.


            \item The proof of the bigness of the monodromy group for
            complete intersections is in line with the proof of Theorem
            ~\ref{theorem:general-type-hypersurface}

            Let \(X_b\) be the smooth complete intersection of multidegree
            \((d_1,\ldots, d_c)\) over a general point \(b\in B\). Let 
            \(X_1,\ldots, X_c\) be the hypersurfaces such that 
            \(X_b=X_1 \cap \cdots \cap X_c\). We may assume the complete
            intersection \(Y \coloneq X_2 \cap \cdots \cap X_{c-1}\)
            is smooth. Then \(X_b\) is a hyperplane section in \(Y\). 
            Let \(D\) be a Lefschetz pencil of hyperplane sections in 
            \(Y\) such that \([X_b]\in D\). The set of hyperplane sections 
            that admit ordinary double points is a finite subset \(S\) in 
            \(D\). Let \(U\coloneqq D-S\). The fundamental group \(\pi_1(U, 0)\)
            has a monodromy action on the vanishing cohomology. For 
            complete intersections, we can compare the vanishing cohomology 
            and the primitive cohomology using~\cite[\S 2.3.3]{Voi-Hodge-II}.
            Since the primitive cohomology of a projective space vanishes, 
            we assert inductively that these two cohomologies for a 
            complete intersection coincde. In the rest, we use
            \(\mathrm{H}^{n-c}_{\mathrm{prim}}(X_b; \mathbb{Q}_{\ell})\)
            in place of the vanishing cohomology.

            The monodromy action of \(\pi_1(U, 0)\) factors through the 
            monodromy action of \(\pi_1(B, b)\) via the natural inclusion 
            \((U, 0)\hookrightarrow (B, b)\). Therefore it suffices to prove 
            the monodromy group of \(\pi_1(U, 0)\) is as big as possible.

            Let \(\mathrm{H}_{\ell}\coloneqq \mathrm{\mathrm{H}}^{n-c}_{\mathrm{prim}}(X_b; \mathbb{Q}_{\ell})\),
            \(\psi\) the intersection form on \(\mathrm{H}_{\ell}\).
            If \(n-c\) is odd, the geometric monodromy group 
            \(G_{\mathrm{geom}}\subset \Aut(\mathrm{H}_{\ell}, \psi)\) 
            is the symplectic group \(\mathrm{Sp}(\mathrm{H}_{\ell}, \psi)\)~\cite[Th\'eor\`em 5.10]{Deligne:Weil-I}.
            If \(n-c\) is even, then \(G_{\mathrm{geom}}\) is either 
            the full orthogonal group \(\mathrm{O}(\mathrm{H}_{\ell}, \psi)\) 
            or a finite subgroup of \(\mathrm{O}(\mathrm{H}_{\ell}, \psi)\)~\cite[Th\'eor\`em 4.4.1]{Deligne:Weil-II}.
            
            If \(G_{\mathrm{geom}}\) is finite, the \(p\)-adic Newton polygon for \(X_b\) 
            is a straight line~\cite[Thm. 11.4.9]{Katz-Sarnak:monodromy}.
            Illusie~\cite{Illusie:ord-ci} proved that the Newton polygon
            of a general complete intersection conincides with its Hodge
            polygon. The Hodge polygon is a straight line if and only if 
            the Hodge numbers \(h^{i, n-c-i}\) of \(X_b\) all vanishes 
            except for \(i=\frac{n-c}{2}\). By~\cite[Expos\'e XI]{SGA7-II},
            such siutation arises only when \((d_1,\ldots,d_c; n)\) 
            falls into one of the following 
            \begin{itemize}
                \item \((2; n)\), hyperquadrics;
                \item \((3; 3)\), cubic surfaces;
                \item \((2,2; n)\) and \(n\) is even, even-dimensional
                complete intersections of two quadrics.
            \end{itemize}
            None of the three cases fits the conditions 
            \(d\geq 3\), \(r\geq 2\), or \(d=2\), \(r\geq 3\).
    \end{enumerate}

Therefore, Theorem~\ref{theorem:mere} assures that \(\Aut_L(Y)\)
is isomorphic to \(\{1\}\) or \(\mathbb{Z}/2\mathbb{Z}\). The second possibility will 
be ruled out by Proposition~\ref{proposition:no-invoultions}.
This completes the proof.
\end{proof}

\begin{proposition}
\label{proposition:no-invoultions}
    Let \(k\) be an algebraically closed field of \(\mathrm{char}(k)\neq 2\).
    Let \(X\subset \mathbb{P}^n_k\) be a complete intersection of 
    multidegree \(\underbrace{(d,\ldots, d)}_{r}\) with \(r < n\). 
    Assmuing \(d\geq 3\), or \(d=2, r\geq 3\), then a generic $X$ has 
    no linear automorphisms of order $2$.
\end{proposition}

\begin{proof}
Let \(V\) be a \(k\)-vector space of dimension \(n+1\) with
\(\mathbb{P}^n_k=\mathbb{V}\). The defining polynomials of the 
codimension \(r\) complete intersection \(X\subset \mathbb{P}^n_k\) 
span an \(r\)-dimensional subspace \(W_X\subset \Sym^d V^*\). 
Let \(\sigma\in \Aut_L(X)\) be a linear automorphism, then \(W_X\) is 
\(\sigma\)-invariant. Our goal is to show that a generic $r$-dimensional 
subspace in \(\Sym^d V^*\) is not stabilized by any involution
\(\sigma\neq \pm \operatorname{Id}\).

Let \(\mathbb{G}(r,\Sym^d V^*)\) be the Grassmannian of \(r\)-dimensional 
subspaces in \(\Sym^d V^*\). For an element 
\(g\in G \coloneq\operatorname{GL}(V)\), we define the subset 
    \[
    G_r(g):=\{ [U]\in \mathbb{G}(r,\Sym^d V^*) ~|~ g\cdot U=U \} 
    \]
of the \(g\)-stabilized subspaces in \(\Sym^d V^*\). If two elements 
\(g, h\in G\) are conjuagte, then \(G_r(g)\) is isomorphic to \(G_r(h)\) 
respectively.

Let $\sigma\in G:=\operatorname{GL}(V)$ be an involution. Denote by
\(\operatorname{Cl}_G(\sigma)\) the conjugacy class of \(\sigma\) 
in \(G\). Define the incidence variety
    \[
    \mathcal{I}_\sigma:=\{(g, [U])\in \operatorname{Cl}_G(\sigma)\times
    \mathbb{G}(r, \Sym^d V^*)~|~ g\cdot U=U\}.
    \]
The projection \(\mathcal{I}_\sigma\to \operatorname{Cl}_G(\sigma)\) 
is a fibration, with each fiber isomorphic to \(G_r(\sigma)\).

Given that \(\mathrm{char}(k)\neq 2\), involutions in \(G\) are similar
to the diagonal matrices of order \(2\). Therefore, to establish our 
assertion, it suffices to prove
\[
\dim \mathbb{G}(r, \Sym^d V^*)>\dim \mathcal{I}_\sigma=\dim \operatorname{Cl}_G(\sigma)+\dim G_{r}(\sigma)
\]
for all diagonal matrices \(\sigma\neq \pm \operatorname{Id}\) of order \(2\).

The action of \(\sigma\) on \(V\) induces a decomposition 
\(V\cong V_+\oplus V_-\) with eigenvalues \(\pm 1\). Then the 
symmetric tensor $\Sym^d V^*$, under the action of $\sigma$,
decomposes into eigenspaces
    \[
    S_+=\bigoplus_{j \text{~even}} S^{d-j} V_+\otimes S^j V_-, ~S_-=
    \bigoplus_{j\text{~odd}} S^{d-j}V_+\otimes S^j V_-.
    \]
Let us set
\begin{align*}
\dim V_+=n_1, \dim V_-=n_2, ~n_1+n_2=n+1, n_1, n_2 > 0,\\
\dim S_+=e_1, \dim S_-=e_2, ~e_1+e_2=\binom{n+d}{d}.
\end{align*}
Suppose that \(W\subset \Sym^d V^*\) is a subspace stabilized
by \(\sigma\). Then \(W\) admits a decomposition \(W_{+}\oplus W_{-}\) 
where \(W_+\subset S_+, W_-\subset S_-\). Hence the set \(G_r(\sigma\)
consists of pairs of subspaces in \(S_+\) and \(S_-\). Let 
\(f_1:=\dim W_{+}\), \(f_2:=\dim W_{-}\). The dimension of \(G_r(\sigma)\) 
is bounded by
\[
\operatorname{max}\{f_1(e_1-f_1)+f_2(e_2-f_2)~|~ 0\leq f_i\leq e_i, f_1+f_2=r\}.
\]

Let \(Z_G(\sigma) < G\) denote the center of \(\sigma\). As a diagonal
matrix \(\sigma\) with \(\mathrm{sign}(\sigma)=(n_1,n_2)\), it is direct
to see \(\dim Z_G(\sigma)=n_1^2+n_2^2\). By the formula
\[
\dim \operatorname{Cl}_G(\sigma)+\dim Z_G(\sigma)=\dim G,
\]
we have \(\dim \operatorname{Cl}_G(\sigma)=2n_1n_2\). Then 
\(\dim \mathbb{G}(r, \Sym^d V^*)-\dim \mathcal{I}_\sigma\) is equal to
\[
\operatorname{min}\{f_2(e_1-f_1)+f_1(e_2-f_2)-2n_1n_2~|~ 0\leq f_i\leq e_i, f_1+f_2=r\}.
\]
Let us show that \(f_2(e_1-f_1)+f_1(e_2-f_2)-2n_1n_2\) are positive 
numbers within the cases \(d\geq 3\) and \(d=2, r\geq 3\).

\begin{itemize}
\item For \(d\geq 3\), we consider the two subcases:

\textbf{(a)} Assuming $f_1\geq f_2>0$, we have the inequality
\[
f_2(e_1-f_1)+f_1(e_2-f_2)-2n_1n_2\geq f_2(e_1+e_2-2f_1)-2n_1n_2.
\]
Observe that $n_1n_2\leq (\frac{n+1}{2})^2$ and $f_1< r<n$. Therefore,
\begin{align*}
f_2(e_1+e_2-2f_1)-2n_1n_2&> f_2(e_1+e_2-2r)-\frac{(n+1)^2}{2}\\
&> f_2(\binom{n+d}{d}-2n)-\frac{(n+1)^2}{2}.
\end{align*}
The term $\binom{n+d}{d}$ increases with $d$. For $d=3$, it is easy 
to verify that
\[
\binom{n+3}{3}-2n-\frac{(n+1)^2}{2}>0, \forall n\geq 3.
\]

\textbf{(b)} Assuming $f_2=0,~f_1=r$, we have
\[
f_2(e_1-f_1)+f_1(e_2-f_2)-2n_1n_2=re_2-2n_1n_2.
\]
The dimension \(e_2 = \dim S_-\), which depends on the integers
$n_1, n_2$ and $d$, also increases with \(d\). For $d=3$, we get
\(e_2=\binom{n_2+2}{3}+n_2\cdot \binom{n_1+1}{2}\). This leads to
\begin{align*}
    re_2-2n_1n_2&=r\cdot \binom{n_2+2}{3}+rn_2\cdot \binom{n_1+1}{2}-2n_1n_2\\
    &\geq 2\cdot \binom{n_2+2}{3}+n_2n_1(n_1+1)-2n_1n_2.
\end{align*}
Hence $re_2-2n_1n_2>0$ unless $n_1=1, n_2=0$. However, thess values 
violates the condition \(n_1+n_2=n+1 > 2\). Swapping 
$f_1$ and $f_2$ does not affect our argument. Thus the assertion follows.

\item For \(d=2, r\geq 3\), we get
\[
e_1=\binom{n_1+1}{2}+\binom{n_2+1}{2}, ~e_2=n_1n_2.
\]
Therefore the experssion \(f_2(e_1-f_1)+f_1(e_2-f_2)-2n_1n_2\) equals
    \[
    f_2(e_1-2f_1)+(f_1-2)n_1n_2.
    \]
We claim that \(e_1-2f_1 > 0\). Since \(n_1+n_2-2=n-1 \geq r\geq f_1\), 
it suffices to prove that \(e_1 > 2(n_1+n_2-2)\). We have
\begin{align*}
     &\binom{n_1+1}{2}+\binom{n_2+1}{2}-2(n_1+n_2-2)\\ 
    =& \frac{1}{2}(n_1^2+n_2^2-3(n_1+n_2)+8) \\
    =& \frac{1}{2}((n_1-\frac{3}{2})^2+(n_2-\frac{3}{2})^2+\frac{7}{2}) >0.
\end{align*}
Furthermore, \(e_1 > 2(n_1+n_2-2)\) ensures \(e_1 > 4\) since 
\(n_1+n_2 = n+1 > r+1\geq 4\). When \(f_1 \geq 2\), it follows that 
\(f_2(e_1-2f_1)+(f_1-2)n_1n_2\) is positive.

Now consider the specific cases \((f_1, f_2) = (0, r)\) and \((1, r-1)\), 
then the expression \(f_2(e_1-2f_1)+(f_1-2)n_1n_2\) becomes \(re_1-2n_1n_2\) 
and \((r-1)(e_1-2)-n_1n_2\) repsectively. We claim that \(e_1-2>n_1n_2\). 
In fact, the condition \(n_1+n_2 >4\) implies that
   \[
   \binom{n_1+1}{2}+\binom{n_2+1}{2}-n_1n_2-2=\frac{1}{2}((n_1-n_2)^2+n_1+n_2-4) > 0.
   \]
Using \(r \geq 3\), it is easy to see that both \(re_1-2n_1n_2\) and \((r-1)(e_1-2)-n_1n_2\)
are positive numbers.
\end{itemize}
\end{proof}

\end{document}